\documentclass[12pt]{amsart}
\usepackage{amsfonts,amsmath,amscd,amssymb,paralist,amsthm}
\usepackage{mathrsfs}
\usepackage{amsxtra}
\usepackage[mathscr]{eucal}
\usepackage{graphicx}
\usepackage{latexsym,graphicx,verbatim,nameref}
\usepackage[numbers,sort&compress]{natbib}
\usepackage[all,cmtip]{xy}
\usepackage{hyperref}
\usepackage{fancyhdr}
\usepackage{color}
\usepackage{tikz}
\usepackage{tikz-cd}
\usepackage{caption}
\usepackage{subfigure}
\usetikzlibrary{arrows,scopes,matrix}
\numberwithin{equation}{section}

\theoremstyle{definition}

\makeatletter
\newtheorem*{rep@theorem}{\rep@title}
\newcommand{\newreptheorem}[2]{%
\newenvironment{rep#1}[1]{%
 \def\rep@title{#2 \ref{##1}}%
 \begin{rep@theorem}}%
 {\end{rep@theorem}}}
\makeatother

\newreptheorem{theorem}{Theorem}
\newreptheorem{corollary}{Corollary}
\newreptheorem{lemma}{Lemma}

\newtheorem{theorem}{Theorem}[section]
\newtheorem{corollary}[theorem]{Corollary}
\newtheorem{lemma}[theorem]{Lemma}
\newtheorem{proposition}[theorem]{Proposition}
\newtheorem*{theorem*}{Theorem}
\newtheorem*{proposition*}{Proposition}

\newtheorem{definition}[theorem]{Definition}
\newtheorem{remark}[theorem]{Remark}

\newtheorem*{claim*}{Claim}

\newtheorem*{conjecture*}{Conjecture}
\newtheorem*{observation*}{Observation}
\newtheorem*{question*}{Question}

\begin{document}

\title{Commuting circle diffeomorphisms with their derivatives having mixed moduli of continuity}

\date{\today}

\author{Hui Xu \& Enhui Shi}

\address[E. Shi]{School of Mathematical Sciences, Soochow University, Suzhou 215006, P. R. China}
\email{ehshi@suda.edu.cn}

\address[H. Xu]
{\noindent School of mathematical and sciences, Soochow University, Suzhou, 215006, P.R. China}
\email{mathegoer@163.com}

\maketitle

\begin{abstract}
Let $d\geq 2$ be an integer and let $\omega_1,\cdots ,\omega_d$ be moduli of continuity in a specified class which  contains the moduli of H\"{o}lder continuity. Let $f_k$, $k\in\{1,\cdots,d\}$, be $C^{1+\omega_k}$ orientation preserving diffeomorphisms of the circle and $f_1,\cdots, f_d$ commute with each other. We prove that if the rotation numbers of $f_k$'s are independent over the rationals and $\omega_1(t)\cdots\omega_d(t)=t\omega(t)$ with $\lim_{t\rightarrow 0^+}\omega(t)=0$, then $f_1,\cdots,f_d$ are simultaneously (topologically) conjugate to rigid rotations.
\end{abstract}

\section{Introduction}
The classical Poincar\'{e}'s classification theorem points out that every orientation preserving  homeomorphism of the circle with irrational rotation number is topologically conjugate or semiconjugate to a rigid rotation. Furthermore, Denjoy \cite{D1,D2} showed that it must be conjugate by adding differentiability of the homeomorphism. Precisely, for a circle diffeomorphism $f$ with irrational rotation number, if $Df$ has bounded variation, then $f$ is conjugate to a rigid rotation. However, for any $0<\tau<1$, there exists a circle diffeomorphism $f$ such that $Df$ is $\tau$-H\"{o}lder continuous and $f$ has wandering intervals.

The smoothness of the conjugacy between a circle diffeomorphism $f$ and a rigid rotation was intensively studied. If $f$ is analytic,  Arnold \cite{A} showed the analyticity of the conjugacy under the assumption that $f$ is sufficiently close to a rotation and its rotation number satisfies certain Diophantine condition.  Moser \cite{M1} obtained the local result in the smooth case.  Later Herman \cite{H} proved the global version: if $f$ is $k$ times differentiable and its rotation number lies in some set $A$ of full measure, then the conjugacy is $k-1-\varepsilon$ times differentiable for any $\varepsilon>0$, and if $f$ is analytic then the conjugacy is analytic. Yoccoz \cite{Y} showed the global result for all Diophantine numbers and the result was sharpen by Katznelson and Ornstein \cite{KO}.

Moser \cite{M2} considered the smoothness of the simultaneous conjugacy between a finite family of commuting circle diffeomorphisms and rigid rotations and obtained the local results by KAM method.  Fayad and Khanin  \cite{FK} established the global version: if the rotation numbers of several commuting $C^{\infty}$ circle diffeomorphisms satisfy the Diophantine condition, then they are  smoothly conjugate to rigid rotations  simultaneously.

Similar to the classical Denjoy Theorem, Kleptsyn and Navas \cite{KN} determined when several commuting circle diffeomorphisms are simultaneously topologically conjugate to rigid rotations. They showed the following theorem.
\begin{theorem}\label{Navas}
Let $d\geq 2$ be an integer. Let $f_{1},\cdots,f_d$ be commuting circle diffeomorphisms such that $f_k$ is $C^{1+\tau_k}$ for each $k\in\{1,\cdots,d\}$, where $0<\tau_k<1$. If $\tau_1+\cdots+\tau_d>1$ and the rotation numbers of $f_k$'s are independent over the rationals, then they are simultaneously topologically conjugate to rotations.

Conversely, given $\alpha_1,\cdots,\alpha_d\in \mathbb{T}$ which are independent over the rationals and $\tau_1,\cdots,\tau_d\in(0,1)$. If $\tau_1+\cdots+\tau_d<1$, then there exist commuting circle diffeomorphisms $f_1,\cdots,f_d$ such that $f_k$ is $C^{1+\tau_k}$ and the rotation number of it is $\alpha_k, k\in \{1,\cdots,d\}$.
\end{theorem}

The main purpose of this paper is to generalize their results by replacing the H\"{o}lder condition by general modulus of continuity. For technical reasons, we introduce the notion of consistency for a family of moduli of continuity in the next section and obtain the following result.
\begin{theorem}\label{Thm}
Let $d\geq 2$ be an integer and let $\omega_1,\cdots,\omega_d$ be concave moduli of continuity with $\omega_1(t)\cdots \omega_d(t)=t\omega(t)$ for some function $\omega(t)$ satisfying that $\lim_{t\rightarrow 0^{+}}\omega(t)=0$. Suppose that $\omega_1,\cdots,\omega_d$ satisfy the consistency condition.
If $f_k, k\in \{1,\cdots, d\}$, are respectively $C^{1+\omega_k}$ commuting circle diffeomorphisms and the rotation numbers of which are independent over the rationals, then they are simultaneously topologically conjugate to rotations.
\end{theorem}

Here we should remark that the proof idea of the above theorem comes from \cite{KN}. The consistency condition can be  implied by some more simpler conditions,
and we get the following corollary.

\begin{corollary}\label{consistency}
Let $d\geq 2$ be an integer and $\omega_1,\cdots,\omega_d$ be concave moduli of continuity satisfying that $\omega_1(t)\cdots \omega_d(t)=t\omega(t)$ with
\begin{equation}\label{consistency eq0}
\lim_{t\rightarrow 0^{+}}\omega(t)=0~\text{and}~\omega(t_1t_2)\ll \omega(t_1)\omega(t_2), \forall t_1,t_2\in [0,1].
\end{equation}
Suppose that $\omega_1,\cdots,\omega_d$ and the moduli of H\"{o}lder continuity compose a comparable family of moduli of continuity. If $f_k, k\in \{1,\cdots, d\}$, are respectively $C^{1+\omega_k}$ circle diffeomorphisms which do commute and the rotation numbers of which are independent over the rationals,  then they are simultaneously topologically conjugate to rotations.
\end{corollary}

The paper is organized as follows. In section 2, we show some definitions and lemmas. In section 3, we prove the main theorem. In the
last section, we proved Corollary \ref{consistency}.
\section{Some Definitions and Lemmas}
For a continuous function $f$ of the circle, the\emph{ modulus of continuity} of $f$ is defined by
\begin{displaymath}
\omega_f(t)=\sup_{|x-y|\leq t} |f(x)-f(y)|.
\end{displaymath}
Then $\omega_f(t)$ is a continuous function on $[0,+\infty)$ and has the following properties:
\begin{itemize}
  \item [(1)] $\omega_f$ is monotonic nondecreasing;
  \item [(2)] $\omega_f(0)=0$ and $\omega_f(t)=1$ for any $t\geq 1$;
  \item [(3)] $\omega_f(t_1+t_2)\leq \omega_f(t_1)+\omega_f(t_2)$;
  \item [(4)] $\omega_f(k t)\leq k\omega_f(t)$ for $k\in\mathbb{N}$ and $t>0$;
  \item [(5)] $\omega_f(r t)\leq ([r]+1)\omega_f(t)$ for $r\geq 0$.
\end{itemize}

\indent Note that if $\omega_f(t)=O(t^{\alpha})$, then $f$ is called {\it $\alpha$-{\rm H\"{o}lder} continuous}. In order to study the modulus of continuity of a function $f$, we often choose some standard moduli to measure the continuity of $f$. The modulus of H\"{o}lder continuity are often chosen.

\begin{definition}
A \textit{modulus of continuity} is a continuous function $w:[0,1]\rightarrow [0,+\infty)$ which is  strictly increasing  and $w(0)=0$. We say a continuous function $f$ of the circle is {\it $\omega$-continuous} if
\begin{displaymath}
\sup_{x\neq y}\frac{|f(x)-f(y)|}{\omega(|x-y|)}<+\infty.
\end{displaymath}
In this case, we say that $\sup_{x\neq y}\frac{|f(x)-f(y)|}{\omega(|x-y|)}$ is the {\it $\omega$-constant} of $f$.
We say $f$ is of {\it class $C^{1+\omega}$} if $f$ is $C^1$ and the derivative of $f$ is $\omega$-continuous.
\end{definition}

In order to compare the moduli of continuity of distinct functions, we often choose some comparable moduli of continuity.
\begin{definition}
Let $\omega_1$ and $\omega_2$ be moduli of continuity. We say that {\it $\omega_1$ is stronger than $\omega_2$}, if there exists $\delta>0$ such that $\omega_1(x)\leq \omega_2(x),\forall x\in(0,\delta)$. In this case, we call they are \textit{comparable}. We say a family $\Omega$ of moduli of continuity is comparable, if there exists $\delta>$ such that for any $\omega,\omega'\in\Omega$, $\omega(x)\geq\omega'(x),\forall x\in(0,\delta)$ or $\omega(x)\leq\omega'(x),\forall x\in(0,\delta)$.
\end{definition}
Let $(x_n)_{n\geq 1}$ and $(y_n)_{n\geq 1}$ be two sequences of positive real numbers. We write $x_n\ll y_n$ if there exists a constant $C>0$ such that
\begin{equation*}
x_n\leq C y_n, ~\forall n\geq 1.
\end{equation*}

\begin{definition}
Let $\omega_1,\cdots,\omega_n$ be moduli of continuity. We call them  \textit{consistent} if there exist strictly increasing sequences $(X_{1,m})_{m\geq 1},\cdots$, $(X_{n,m})_{m\geq 1}$  of positive integers such that for each $k\in\{1,\cdots,n\}$,
\begin{displaymath}
\omega_k\left(\frac{1}{\prod_{j=1}^n X_{j,m}} \right)\ll \prod_{j=1}^n \omega_j\left( \frac{1}{X_{k,m}}\right).
\end{displaymath}
\end{definition}

\begin{remark}
The consistency is a technical condition. It is easy to check that the usual moduli of continuity satisfy the consistency. Taking $\omega_1(x)=x^{\alpha_1},\cdots,\omega_d(x)=x^{\alpha_d}$ as example, we can choose
\begin{equation*}
X_{1,m}=\left[ 2^{\alpha_{1}m}\right],\cdots,X_{1,m}=\left[ 2^{\alpha_{d}m}\right],
\end{equation*}
for $m$ large enough.
\end{remark}

\begin{proposition}
Let $d\geq 2$ be an integer and $\alpha_1,\cdots,\alpha_d\in(0,1), \varepsilon_1,\cdots$, $\varepsilon_d\geq0$. The moduli of continuity $\omega_1(x)=x^{\alpha_1}\left(\log\frac{1}{x}\right)^{\varepsilon_1},\cdots, \omega_d(x)=x^{\alpha_d}\left(\log\frac{1}{x}\right)^{\varepsilon_d}$ for $x$ small are consistent.
\end{proposition}
\begin{proof}
Let $a>1$ and $x_{j,m}=\frac{1}{a^{\alpha_jm}}, j\in\{1,\cdots,d\}$. Then for any $j\in\{1,\cdots,d\}$, we have
\begin{eqnarray*}
\omega_j(x_{1,m}\cdots x_{d,m})&=&\frac{1}{a^{(\alpha_1+\cdots+\alpha_d)\alpha_j m}}\left(\log a^{(\alpha_1+\cdots\alpha_d)m}\right)^{\varepsilon_j}\\
&=&\frac{1}{a^{(\alpha_1+\cdots+\alpha_d)\alpha_j m}}(\alpha_1+\cdots+\alpha_d)^{\varepsilon_j}\left(m\log a\right)^{\varepsilon_j},
\end{eqnarray*}
and
\begin{eqnarray*}
\omega_1(x_{j,m})\cdots \omega_d(x_{j,m})&=&\frac{1}{a^{\alpha_1\alpha_j m}}\left(\log a^{\alpha_jm}\right)^{\varepsilon_1}\cdots \frac{1}{a^{\alpha_d\alpha_j m}}\left(\log a^{\alpha_jm}\right)^{\varepsilon_d}\\
&=&\frac{1}{a^{(\alpha_1+\cdots+\alpha_d)\alpha_j m}}(\alpha_1)^{\varepsilon_1}\cdots(\alpha_d)^{\varepsilon_d}\left(m\log a\right)^{\varepsilon_1+\cdots \varepsilon_d}.
\end{eqnarray*}
It is easy to see that
\begin{equation*}
\omega_j(x_{1,m}\cdots x_{d,m})\ll \omega_1(x_{j,m})\cdots \omega_d(x_{j,m}),~\forall m\in\mathbb{N}.
\end{equation*}
Set $X_{j,m}=[a^{\alpha_jm}]$ for $m$ large enough. Then, by the continuity, we also have
\begin{equation*}
\omega_j\left(\frac{1}{X_{1,m}\cdots X_{d,m}}\right)\ll \omega_1\left(\frac{1}{X_{j,m}}\right)\cdots \omega_d\left(\frac{1}{X_{j,m}}\right),~ m\in\mathbb{N}.
\end{equation*}
This shows that $\omega_1,\cdots,\omega_d$ are consistent.
\end{proof}

\begin{lemma}\label{modulus of continuity for compsition}
Let $\varphi$ and $\psi$ be continuous functions of the circle. If $\varphi$ is Lipschitz-continuous and $\psi $ is $\omega$-continuous, then $\varphi\circ\psi$ is also $\omega$-continuous.
\end{lemma}
\begin{proof} The $\omega$-continuity of $\varphi\circ\psi$ can be seen as follows:
\begin{eqnarray*}
&&\sup_{x\neq y}\frac{|\varphi\circ\psi(x)-\varphi\circ\psi(y)|}{\omega(|x-y|)}\\
&=&\sup_{x\neq y}\frac{|\varphi\circ\psi(x)-\varphi\circ\psi(y)|}{|\psi(x)-\psi(y)|}\frac{|\psi(x)-\psi(y)|}{\omega(|x-y|)}\\
&\leq&\sup_{x\neq y}\frac{|\varphi\circ\psi(x)-\varphi\circ\psi(y)|}{|\psi(x)-\psi(y)|}\sup_{x\neq y} \frac{|\psi(x)-\psi(y)|}{\omega(|x-y|)} <+\infty.
\end{eqnarray*}
\end{proof}

\begin{remark}
If $g$ is an orientation preserving circle diffeomorphism, then there exist constants $C_1,C_2>0$ such that $C_1\leq Dg_k(x)\leq C_2, \forall x\in \mathbb{T}^1$. Note that the function $\log(x), x\in [C_1,C_2]$ is Lipschitz. Then, by Lemma \ref{modulus of continuity for compsition}, we know that $\log Dg$ has the same modulus of continuity with $Dg$.
\end{remark}

\begin{lemma}\label{fixed pt}
Let $\omega_1,\cdots, \omega_l$ be concave moduli of continuity and $g_k,k\in\{1,\cdots, l\}$ be orientation preserving circle diffeomorphisms which are respectively of class $C^{1+\omega_k}$. Let $C_k$ denote the $\omega_k$-constant of $\log Dg_k$ and $C=\max\{ C_1,\cdots, C_l\}$.  Given $n_0\in \mathbb{N}$, for each $n\leq n_0$, let us choose $k_n\in\{1,\cdots l\}$ and  for a fixed interval $I$, let us choose a constant $S>0$ such that
\begin{displaymath}
\sum_{n=0}^{n_0-1} \omega_{k_{n+1}}(g_{k_n}\cdots g_{k_1}(I)) \leq S.
\end{displaymath}
If there exists $n\leq n_0$ such that $g_{k_n}\cdots g_{k_1}(I)$ is contained in the $L$-neighborhood of $I$ but does not intersect $I$, then $g_{k_n}\cdots g_{k_1}$ has a fixed point, where $L:= \frac{|I|}{2\exp(2CS)}$.
\end{lemma}
\begin{proof}
Let $J$ be the closed $2L$-neighborhood of $I$ and $I_1, I_2$ the connected components of $J\setminus I_1$ to the complement of $I$. By induction on $j\in\{0,\cdots,n_0\}$, we will prove the following properties:
\begin{itemize}
  \item [($A_j$)] $|g_{k_j}\cdots g_{k_1}(I_1)|\leq |g_{k_j}\cdots g_{k_1}(I)|$;
  \item [($B_j$)] \begin{math} \sup_{x,y\in I\cup I_1} \frac{D(g_{k_j}\cdots g_{k_1})(x)}{D(g_{k_j}\cdots g_{k_1})(y)} \leq \exp(2CS) \end{math}.
\end{itemize}
The properties $(A_0)$ and $(B_0)$ are trivial. Suppose that $(A_i)$ and $(B_i)$ hold for every $i\in\{0,\cdots,j-1\}$. Then for any $x,y\in I\cup I_1$, we have

\begin{eqnarray*}
&&\left|\log \frac{D(g_{k_j}\cdots g_{k_1})(x)}{D(g_{k_j}\cdots g_{k_1})(y)}\right| \\
&\leq & \sum_{i=0}^{j-1}\left| \log(D(g_{k_j}\cdots g_{k_1})(x))-\log(D(g_{k_j}\cdots g_{k_1})(y))\right| \\
&\leq& \sum_{i=0}^{j-1} C_{k_{i+1}}\omega_{k_{i+1}} \left( |D(g_{k_j}\cdots g_{k_1})(x)-D(g_{k_j}\cdots g_{k_1})(y) | \right)\\
&\leq& C \sum_{i=0}^{j-1} \omega_{k_{i+1}} \left( |D(g_{k_j}\cdots g_{k_1})(I)|+|(D(g_{k_j}\cdots g_{k_1})(I_1) | \right).
\end{eqnarray*}

Since  $\omega(t)$  is increasing and $|(D(g_{k_j}\cdots g_{k_1})(I_1) |\leq |(D(g_{k_j}\cdots g_{k_1})(I)|$ by induction hypothesis, the righthand side of the above inequality is bounded by
\begin{eqnarray*}
&& C \sum_{i=0}^{j-1} \omega_{k_{i+1}}\left( 2 |D(g_{k_j}\cdots g_{k_1})(I)|\right)\\
&\leq& 2C\sum_{i=0}^{j-1} \omega_{k_{i+1}} \left( |D(g_{k_j}\cdots g_{k_1})(I)|\right)\\
&\leq& 2CS.
\end{eqnarray*}
The first inequality above is followed that $\omega(t)\geq \frac{1}{2}\omega(0)+\frac{1}{2}\omega(2t)\Longrightarrow \omega(2t)\leq 2\omega(t)$, since $\omega(t)$ is concave.
Thus $(B_j)$ follows.  By the Mean Value Theorem, there exist $x_j\in I_1$ and $y_j\in I$ such that
\begin{displaymath}
|D(g_{k_j}\cdots g_{k_1})(I_1)|=D(g_{k_j}\cdots g_{k_1})(x_j)| I_1|,
\end{displaymath}
and
\begin{displaymath}
 |D(g_{k_j}\cdots g_{k_1})(I)|=D(g_{k_j}\cdots g_{k_1})(y_j)| I|.
\end{displaymath}
 Then by $(B_j)$, we obtain $(A_j)$, since 
 \begin{displaymath}
 \frac{|D(g_{k_j}\cdots g_{k_1})(I_1)|}{|D(g_{k_j}\cdots g_{k_1})(I)|}= \frac{D(g_{k_j}\cdots g_{k_1})(x_j)|I_1|}{D(g_{k_j}\cdots g_{k_1})(y_j)|I|}\leq\exp(CS)\frac{|I_1|}{|I|}\leq 1.
 \end{displaymath}
Similarly, we have the analogous arguments for $I_2$.\\
Suppose that  $g_{k_n}\cdots g_{k_1}(I)$ is contained in the $L$-neighborhood of $I$ but does not intersect $I$. Then $(A_n)$ implies that $g_{k_n}\cdots g_{k_1}(J)\subset J$. Hence  $g_{k_n}\cdots g_{k_1}$ has a fixed point in $J$.
\end{proof}

\begin{lemma}\label{comb. lemma}
Let $l_{i,j}$ be positive real numbers with $i\in\{1,\cdots,m\}$  and $j\in\{1,\cdots,n\}$. Suppose that $\sum_{i=1}^m\sum_{j=1}^n l_{i,j}\leq 1$. Let $\omega:[0,1]\rightarrow \mathbb{R}_{+}$ be an increasing and  concave function. Then there exists $k\in\{1,\cdots,n\}$ such that
\begin{displaymath}
\sum_{i=1}^m \omega(l_{i,j})\leq  m\omega\left(\frac{1}{mn}\right).
\end{displaymath}
Moreover, for any $A>0$, there is a proportion of indices $k\in\{1,\cdots,n\}$ no less than $(1-1/A)$ such that
\begin{displaymath}
\sum_{i=1}^m \omega(l_{i,k})\leq   A m\omega\left(\frac{1}{mn}\right).
\end{displaymath}
\end{lemma}
\begin{proof}
For any $k\in\{1,\cdots,n\}$ , by the concavity of $\omega$, we have
\begin{displaymath}
\sum_{i=1}^m \omega(l_{i,j})=m\sum_{i=1}^m \frac{1}{m}\omega(l_{i,j})\leq m\omega\left(\sum_{i=1}^m \frac{l_{i,j}}{m}\right).
\end{displaymath}
Thus
\begin{eqnarray*}
\frac{1}{n} \sum_{k=1}^{n}\left(\sum_{i=1}^m \omega(l_{i,j}) \right) &\leq& \frac{m}{n}\sum_{k=1}^n\omega\left(\sum_{i=1}^m \frac{l_{i,j}}{m}\right)\\
&=&  m\sum_{k=1}^n\frac{1}{n}\omega\left(\sum_{i=1}^m \frac{l_{i,j}}{m}\right)\\
&\leq & m \omega\left(\sum_{k=1}^n\sum_{i=1}^m \frac{l_{i,j}}{mn}\right)\\
&\leq & m\omega\left(\frac{1}{mn}\right).
\end{eqnarray*}
Therefore,  there exists $k\in\{1,\cdots,n\}$ such that
\begin{displaymath}
\sum_{i=1}^m \omega(l_{i,j})\leq  m\omega\left(\frac{1}{mn}\right).
\end{displaymath}
In order to prove the second part, we define
\begin{displaymath}
E:=\left\{k\in\{1,\cdots,n\}:~ \sum_{i=1}^m \omega(l_{i,k})\geq   A m\omega\left(\frac{1}{mn}\right) \right\}.
\end{displaymath}
In other words, we need to show that
\begin{displaymath}
|E|\geq \left(1-\frac{1}{A}\right)n.
\end{displaymath}
Note that
\begin{eqnarray*}
m\omega\left(\frac{1}{mn}\right)&\geq& \frac{1}{n} \sum_{k=1}^{n}\left(\sum_{i=1}^m \omega(l_{i,j}) \right)\geq \frac{1}{n} \sum_{k\notin E}\left(\sum_{i=1}^m \omega(l_{i,j}) \right)\\
&\geq& \frac{n-|E|}{n}  A m\omega\left(\frac{1}{mn}\right).
\end{eqnarray*}
Thus $|E|\geq \left(1-\frac{1}{A}\right)n. $
\end{proof}

\section{Proof of Theorem\ref{Thm}}

 Suppose that the theorem does not hold. Then there is a wandering interval  with nonempty interior. Let $I$ be a maximal wandering interval. We will search for a sequence  $h_n:=f_{k_n}\cdots f_{k_1}$ satisfying the hypothesis of Lemma \ref{fixed pt}. Then we conclude that $h_n$ has a fixed point. This implies that the rotation number of $h_n$ is zero. Recall that for a commutative group $G$ of circle homeomorphisms, there is a $G$-invariant probability measure. Then the rotation number is a homomorphism from $G$ to $\mathbb{R}/\mathbb{Z}$. But this contradicts to the hypothesis that the rotation numbers of $f_1,\cdots,f_d$ are independent over the rationals.\\

\subsection{The case $d=2$}

Let $l_{i,j}=\left| f_1^{i} f_2^{j}(I) \right|$ . Since $I$ is a wandering interval, $ \sum_{i,j\geq 0}l_{i,j}\leq 1$. \\
\indent In the subsequent, we use $[a,b]$ to denote the integers between $a$ and $b$ including the end if it is an integer.\\
\indent Let $(i_n)_{n\geq 1}$ and $(j_{n})_{n\geq 1}$ be two sequences of non negative integers satisfying $i_0=i_1=j_0=j_1=0$. Consider a sequence of rectangles $R_{m}\subset \mathbb{N}\times \mathbb{N}$ with
$R_0=\{(0,0)\}$ and
\begin{displaymath}
R_{2m+1}=[i_m,i_{m+1}]\times [j_m,j_{m+2}],~R_{2m+2}=[i_m,i_{m+2}]\times [j_{m+1},j_{m+2}].
\end{displaymath}
Denote by $X_m$ and $Y_m$ the number of integers on the horizontal and vertical sides of $R_{m}$ respectively .\\
\indent Applying Lemma \ref{comb. lemma} to $ R_m$ gives us a sequence $(r(m))_{m\geq0}$ of integers such that
\begin{itemize}
  \item  $r(2m+1)\in [i_m,i_{m+1}]$ such that
  \begin{eqnarray*}
  \sum_{j=j_m}^{j_{m+2}} \omega_{2}(l_{r(2m+1),j}) &\leq& Y_{2m+1} \omega_{2}\left(\frac{1}{X_{2m+1}Y_{2m+1}}\right)\\&=&\frac{ \omega_{2}\left(\frac{1}{X_{2m+1}Y_{2m+1}}\right)}{ \omega_{1}\left(\frac{1}{Y_{2m+1}}\right) \omega_{2}\left(\frac{1}{Y_{2m+1}}\right)}  \omega\left(\frac{1}{Y_{2m+1}}\right);
  \end{eqnarray*}
  \item  $r(2m+2)\in [j_{m+1},j_{m+2}]$ such that
  \begin{eqnarray*}
  \sum_{i=i_m}^{i_{m+2}} \omega_{1}(l_{i,r(2m+2)}) &\leq& X_{2m+2} \omega_{1}\left(\frac{1}{X_{2m+2}Y_{2m+2}}\right)\\&=&\frac{ \omega_{1}\left(\frac{1}{X_{2m+2}Y_{2m+2}}\right)}{ \omega_{1}\left(\frac{1}{X_{2m+2}}\right) \omega_{2}\left(\frac{1}{X_{2m+2}}\right)}  \omega\left(\frac{1}{X_{2m+2}}\right).
  \end{eqnarray*}
\end{itemize}
Starting from the origin and following the corresponding horizontal line $y=r(2m+2)$ and vertical line $x=r(2m+1)$, we obtain a path $(i(n),j(n))_{n\geq 0}$ with $i(0)=j(0)=0$ and
\begin{equation*}
(i(n+1),j(n+1))-(i(n),j(n))=(0,1) ~~\text{or}~~(1,0).
\end{equation*}
Moreover, the sum $\sum_{n\geq 0} \omega_{\xi(n)}\left(l_{i(n),j(n)}\right) $is bounded by
\begin{equation*}
\sum_{m\geq 0} \left(\frac{ \omega_{2}\left(\frac{1}{X_{2m+1}Y_{2m+1}}\right)}{ \omega_{1}\left(\frac{1}{Y_{2m+1}}\right) \omega_{2}\left(\frac{1}{Y_{2m+1}}\right)}  \omega\left(\frac{1}{Y_{2m+1}}\right)+ \frac{ \omega_{1}\left(\frac{1}{X_{2m+2}Y_{2m+2}}\right)}{ \omega_{1}\left(\frac{1}{X_{2m+2}}\right) \omega_{2}\left(\frac{1}{X_{2m+2}}\right)}  \omega\left(\frac{1}{X_{2m+2}}\right)\right),
\end{equation*}
where
\begin{equation*}
\xi(n)=\begin{cases}
1, & (i(n+1),j(n+1))=(i(n),j(n))+(1,0) \\
2,&(i(n+1),j(n+1))=(i(n),j(n))+(0,1).
\end{cases}
\end{equation*}
\indent Since $\omega_1$ and $\omega_2$ are consistent, we can choose two strictly increasing sequences of integers $(i_m)_{m\geq 0}$ and $(j_m)_{m\geq 0}$ such that both $X_m$ and $Y_m$ tend to infinity as $m$ goes to infinity and
\begin{equation}\label{eq1}
\omega_{1}\left(\frac{1}{X_{2m+2}Y_{2m+2}}\right)\ll  \omega_{1}\left(\frac{1}{X_{2m+2}}\right) \omega_{2}\left(\frac{1}{X_{2m+2}}\right),
\end{equation}
and
\begin{equation}\label{eq2}
\omega_{2}\left(\frac{1}{X_{2m+1}Y_{2m+1}}\right)\ll  \omega_{1}\left(\frac{1}{Y_{2m+1}}\right) \omega_{2}\left(\frac{1}{Y_{2m+1}}\right).
\end{equation}

Since  $\lim_{t\rightarrow 0^{+}}\omega(t)=0$, there exist a subsequence $m_1<m_2<\cdots$ of positive integers such that
\begin{equation*}
\omega\left(\frac{1}{X_{2m_j+2}}\right) <\frac{1}{j^2}~~\text{and}~~\omega\left(\frac{1}{Y_{2m_j+1}}\right) <\frac{1}{j^2}.
\end{equation*}
Thus we may assume that
\begin{equation*}
\sum_{m\geq 0} \left( \omega\left(\frac{1}{Y_{2m+1}}\right)+\omega\left(\frac{1}{X_{2m+2}}\right)\right)< +\infty.
\end{equation*}
Therefore, there exists a constant $C_1>0$ such that
\begin{equation*}
\sum_{n\geq 0} \omega_{\xi(n)}\left(l_{i(n),j(n)}\right) \leq C_1\sum_{m\geq 0} \left( \omega\left(\frac{1}{Y_{2m+1}}\right)+\omega\left(\frac{1}{X_{2m+2}}\right)\right)=:S< +\infty.
\end{equation*}
For $n\geq 1$, let $k_n=\xi(n-1)\in\{1,2\}$. Then we obtain a sequence $h_n:=f_{k_n}\cdots f_{k_1}$ such that
\begin{equation*}
\omega_{k_{n+1}}\left(f_{k_n}\cdots f_{k_1}(I) \right)=\omega_{\xi(n)}\left(l_{i(n),j(n)}\right).
\end{equation*}
Hence
For $n\geq 1$, let $k_n=\xi(n-1)\in\{1,2\}$. Then we obtain a sequence $h_n:=f_{k_n}\cdots f_{k_1}$ such that
\begin{equation*}
\sum_{n\geq0}\omega_{k_{n+1}}\left(f_{k_n}\cdots f_{k_1}(I) \right)\leq S.
\end{equation*}

In order to apply Lemma\ref{fixed pt}, it suffices to show that there exists some $n\geq 1$ such that $h_n(I)=f_{k_n}\cdots f_{k_1}(I)$ is contained in the $L$-neighborhood of $I$, since $I$ is a wandering interval.\\

Since $f_1$ and $f_2$ are semiconjugate to irrational rotations, if we collapse every connected component of the complement of the minimal invariant Cantor set, then we get a topological circle $\tilde{S^{1}}$ on which $f_1$ and $f_2$ induce minimal homeomorphisms $\tilde{f_1}$ and $\tilde{f_2}$ respectively. Now the $L$-neighborhood of $I$ becomes an interval $V$ with nonempty interior in $\tilde{S^{1}}$. Then there must exist $N\in\mathbb{N}$ such that both $\tilde{f_1}^{-1}(V),\cdots, \tilde{f_1}^{-N}(V)$ and $\tilde{f_2}^{-1}(V),\cdots, \tilde{f_2}^{-N}(V)$ cover the circle $\tilde{S^{1}}$.  Since both $X_m$ and $Y_m$ go to infinity as $m$ tends to infinity, there exists an integer $r\in\mathbb{N}$ such that $k_r=k_{r+1}=\cdots=k_{r+N}$. Converting to the original standard circle, there exist $k,k'\in\{1,\cdots,N\}$ such that both $f_1^{k}f_{k_r}\cdots f_{k_1}(I)$ and $f_2^{k'}f_{k_r}\cdots f_{k_1}(I)$ are contained in the $L$-neighborhood of $I$. This implies that at least one of the intervals $h_{r+1}(I),\cdots,h_{r+N}(I)$ is contained in the $L$-neighborhood of $I$. Thus we complete the proof.

\subsection{The general case}
Let $(l_{i_1,\cdots,i_d})_{i_1,\cdots,i_d\geq 1}$ be a multi-indexed sequence of positive real numbers with $\sum_{i_1,\cdots,i_d\geq 1}l_{i_1,\cdots,i_d}\leq 1$. Let $(x_{1,m})_{m\geq 1}$, $\cdots,(x_{d,m})_{m\geq 1}$ be sequences of strict increasing nonnegative integers and let $R_0=\{(0,\cdots,0)\}$. Consider a sequence $(R_m)_{m\geq 0}$ of rectangles of the form $R_{m}= [0.x_1,m]\times\cdots\times[0,x_{d,m}]$  satisfying that for each $k\in\{1,\cdots,d\}$,
\begin{equation*}
x_{k,m}
\begin{cases}
= x_{k,m-1},& m\not\equiv k\\
> x_{k,m-1},& m\equiv k.
\end{cases}
\end{equation*}
Now for each $m\geq 1$, let $r(m)$ be the unique integer in $\{1,\cdots,d\}$ such that $m\equiv r(m) \mod d$. Denote by $F_m$ the face
\begin{equation*}
[0,x_{1,m}]\times\cdots[0,x_{r(m)-1,m}]\times \{0\}\times[0,x_{r(m)+1,m}]\times\cdots\times[0,x_{d,m}]
\end{equation*}
of $R_m$. Let $A_m>1$ and denote by $E_m$ the set of $(i_1,\cdots,i_{r(m)-1},0,i_{r(m)+1}$, $\cdots,i_d)\in F_m$ such that
\begin{equation*}
 \sum_{j=0}^{x_{r(m),m}} \omega_{r(m)} \left( l_{i_1,\cdots,i_{r(m)-1},0,i_{r(m)+1},\cdots,i_d}\right)\leq A_m X_{r(m),m}  \omega_{r(m)} \left( \frac{1}{\prod_{j=1}^{d}X_{j,m}}\right),
\end{equation*}
where $X_{j,m}:=1+x_{j,m}$. Then by Lemma \ref{comb. lemma}, we have
\begin{equation*}
\frac{|E_m|}{|F_m|}\geq 1-\frac{1}{A_m}.
\end{equation*}
Similar to the proof of the case $d=2$, we need to choose a path of points $(x_1(n),\cdots,x_d(n))_{n\geq 0}$ starting at the origin which is long enough and satisfies
\begin{equation*}
\sum_{n=0}^{P} \omega_{\xi(n)}\left(l_{x_{1}(n),\cdots,x_d(n)}\right)<S,
\end{equation*}
where $P$ denote the length of the path and $S$ is independent of $P$. In order to show the existence of such path, we recall the following lemma which is showed in \cite{KN}.
\begin{lemma}\label{admissible lines}
Denote by $\mathcal{L}(m)$ the set of all lines inside the rectangle $R_m$ in the $r(m)$-direction for each $m\geq 1$. Let $\mathcal{L}'(m)$ be a subset of $\mathcal{L}(m)$ such that$|\mathcal{L}'(m)|\geq  \left(1-\frac{1}{A_m}\right)|\mathcal{L}(m)|$. If there exist $M\in\mathbb{N}$ such that $\sum_{m=1}^M \frac{1}{A_m}<1$, then there exists a sequence of lines $L_m\in \mathcal{L}'(m), m=1,\cdots,M$, such that $L_{m+1}$ intersects $L_m$ for every $0\leq m<M$, where $L_0$ denote the origin.
\end{lemma}

Since $\omega_1,\cdots,\omega_d$ satisfy the consistency condition, we can choose $x_{k,m}$'s such that for each $m\geq 0$,
\begin{equation*}
 \omega_{r(m)}\left( \frac{1}{\prod_{j=1}^{d}X_{j,m}}\right)\ll \prod_{j=1}^d \omega_{j} \left(\frac{1}{X_{j,m}} \right),
\end{equation*}
and since $\lim_{t\rightarrow 0^{+}}\omega(t)=0$, by passing to a subsequence, we may assume that
\begin{displaymath}
\sum_{m\geq 0}\left(\omega \left(\frac{1}{X_{r(m),m}}\right)\right)^{\frac{1}{2}}=:S_j<\infty.
\end{displaymath}
We choose $A_m=2S_{r(m)}/\left(\omega \left(\frac{1}{X_{r(m),m}}\right)\right)^{\frac{1}{2}}$. Then $\sum_{m\geq 0}\frac{1}{A_m}<1$. Now the previous Lemma \ref{admissible lines} provides us a desired sequence of lines $L_m,m\in \{0,\cdots,M\}$ for each $M\in \mathbb{N}$.
Then the sequence of lines $L_m$ induces a finite paths of points $(x_1(n),\cdots,x_d(n))_{n\geq 0}$ starting at the origin satisfying
\begin{equation*}
(x_1(n+1),\cdots,x_d(n+1))=(x_1(n),\cdots,x_d(n))+(0,\cdots,0,\pm 1,0,\cdots,0).
\end{equation*}
Moreover, if we denote by $\ell(M)$ the length of the path, then
\begin{eqnarray*}
\sum_{n=0}^{\ell(M)} \omega_{\xi(n)}\left(l_{x_{1}(n),\cdots,x_d(n)}\right)&\leq& \sum_{m=0}^M A_m X_{r(m),m}\omega_{r(m)}\left( \frac{1}{\prod_{j=1}^{d}X_{r(m),m}}\right)\\
&=& \sum_{m=0}^M A_m \frac{ \omega_{r(m)}\left( \frac{1}{\prod_{j=1}^{d}X_{j,m}}\right)}{ \prod_{j=1}^d \omega_{j} \left(\frac{1}{X_{r(m),m}} \right) } \omega \left(\frac{1}{X_{r(m),m}}\right)\\
&\leq&  A  \sum_{m=0}^M S_{r(m)}\left(\omega \left(\frac{1}{X_{r(m),m}}\right)\right)^{\frac{1}{2}}\\
&\leq&  A  \sum_{m\geq0} S_{r(m)}\left(\omega \left(\frac{1}{X_{r(m),m}}\right)\right)^{\frac{1}{2}}=:S<+\infty,
\end{eqnarray*}
where $\xi(n)$ is the unique index in $\{1,\cdots,d\}$ such that $|x_{\xi(n)}(n+1)-x_{\xi(n)}(n)|=1$, and $A$ is a constant independent of $M$.\\

Now we can prove the general case in a similar way as in the case $d=2$. Let $I$ be a maximal open wandering interval, and define
\begin{equation*}
l_{i_1,\cdots,i_d}= |f_1^{i_1}\cdots f_d^{i_d}(I)|.
\end{equation*}
Let $C:=\max\{C_1,\cdots,C_d,C_1',\cdots,C_d'\}$ with $C_k$ (resp. $C_k'$) being the $\omega_k$-constant of $\log(D f_k)$ (resp. $\log(D f_k^{-1})$) and take $C'$ such that
\begin{equation*}
\omega_k(2t)\leq C' \omega_k(t),~\forall t\geq 0~~\text{and}~~\forall k\in \{1,\cdots,d\}.
\end{equation*}
Now set $L:=\frac{1}{2\exp( 2CS)}$. Let $V$ be the $L$-neighborhood of $I$. Let $\Gamma$ be the semigroup generated by $f_1,\cdots,f_d, f_1^{-1},\cdots,f_d^{-1}$.\\
By the same reason as in the proof of the case $d=2$, there exists $N\in \mathbb{N}$ such that  for any $g\in \Gamma$, there exist $r_1,\cdots,r_d\in\{1,\cdots,N\}$ such that $f_1^{r_1}(g(I)),\cdots , f_d^{r_d}(g(I))$ are all contained in  the $L$-neighborhood of $I$.\\

Now choose $M\in\mathbb{N}$ large enough such that the number of points with integer coordinates in the $L_M$ in the $r(M)$-direction which are all contained in $R_M\setminus R_{M-1}$ exceeds N. Then we can complete the proof in the very way as in the case $d=2$.

\section{Proof of Corollary \ref{consistency}}

\begin{proof}
Let $C>0$ be such that $\omega(t_1t_2)\leq C \omega(t_1)\omega(t_2), \forall t_1,t_2\in [0,1]$. For simplicity, we firstly show the case of $d=2$.\\
By the continuity, it suffices to show that for any neighborhood $V$ of $(0,0)$, there exist $(x,y)\in V$ such that
\begin{equation}\label{consistency eq1}
\omega_1(xy)\leq C\omega_1(x)\omega_2(x),
\end{equation}
and
\begin{equation}\label{consistency eq2}
\omega_2(xy)\leq C\omega_1(y)\omega_2(y).
\end{equation}
 Define
\begin{equation*}
\phi(x):=\sup\{ y\in[0,1]: \omega_1(xy)\leq \omega_1(x)\omega_2(x)\}.
\end{equation*}
Now let $y=\phi(x)$. Then $\omega_{1}(xy)=\omega_1(x)\omega_2(y)$. Thus (\ref{consistency eq2}) is equivalent to
\begin{equation}\label{consistency eq3}
\omega_1(xy)\omega_2(xy)\leq C\omega_1(x)\omega_2(x)\omega_1(y)\omega_2(y).
\end{equation}
Since $\omega_1(t)\cdots \omega_d(t)=t\omega(t)$, (\ref{consistency eq3}) is equivalent to  (\ref{consistency eq0}). It is easy to see that $\lim_{x\rightarrow 0}\phi(x)=0$. Hence $(x,\phi(x))$ is such that both (\ref{consistency eq1}) and (\ref{consistency eq2}) hold.\\

For the general case, it suffices to show that for any  any neighborhood $U$ of the origin, there exist $(x_1,\cdots,x_d)\in U$ such that
\begin{equation}\label{eq3}
\omega_k(x_1\cdots x_d)\leq C^{d-1} \omega_1(x_k)\cdots \omega_d(x_k),~\forall k\in\{1,\cdots,d\}.
\end{equation}

Since $\omega_1,\cdots,\omega_d$ are comparable and $\lim_{t\rightarrow 0^+}\omega(t)=0$, we may assume that there exists a constant $\delta\in(0,1)$ such that for any $ x\in(0,\delta)$,
\begin{equation}\label{comparability}
\omega_d(x)\leq \omega_1(x)\leq\omega_{2}(x)\leq\cdots\leq \omega_{d-1}(x)~\text{and}~\omega(x)\leq 1.
\end{equation}

We will show that for any $x_1\in(0,\delta)$, there exist $x_2,\cdots,x_d\in(0,1)$ such that
\begin{equation}
\omega_k(x_1\cdots x_d)= \omega_1(x_k)\cdots \omega_d(x_k),~\forall k\in\{1,\cdots,d-1\}.
\end{equation}
Note that
\begin{equation*}
0=\omega_1(0)<\omega_1(x_1)\cdots\omega_d(x_1)<\omega(x_1).
\end{equation*}
Thus there exists a unique $t_2\in(0,1)$ such that
\begin{equation}
\omega_1(x_1t_2)=\omega_1(x_1)\cdots\omega_d(x_1),
\end{equation}
since $\omega_1$ is strictly increasing and continuous. By Theorem \ref{Navas}, we may assume that $\omega_1(x)\geq  x^{\frac{1}{2}}, \forall x\in (0,\delta)$. Then we have
\begin{equation*}
(x_1t_1)^{\frac{1}{2}}\leq \omega_1(x_1t_2)=\omega_1(x_1)\cdots\omega_d(x_1)=x_1\omega(x_1)\leq x_1.
\end{equation*}
Hence $t_2\leq x_1$.  Then
\begin{eqnarray*}
\omega_1(t_2)\cdots\omega_d(t_2)&\leq& \omega_1(x_1)\cdots \omega_d(x_1)\\
&=&\omega_1(x_1t_2)\leq \omega_2(x_1t_2)\\
&<&\omega_1(1)\cdots\omega_d(1)=1.
\end{eqnarray*}
Thus, since  $\omega_2$ is strictly increasing and continuous, there exists a unique $x_2\in[t_2,1)$ such that
\begin{equation*}
\omega_2(x_1t_2)=\omega_1(x_2)\cdots\omega_d(x_2).
\end{equation*}
\indent In order to make the process be continued, we have to make subtler estimation. By Theorem \ref{Navas}, we may assume that $\omega_{k}(x)\geq x^{\tau_k}, k\in\{1,\cdots,d\}$, with $1>\tau_d\geq \tau_1\geq\tau_2\geq\cdots\geq\tau_{d-1}\geq 0$ and $\tau_1+\cdots+\tau_d\leq 1$. Provided that we have found $x_1,x_2,\cdots,x_j$ and $t_2,\cdots,t_j$, for $2\leq j\leq d-2$ such that $t_{k}=x_kt_{k-1}$ and
\begin{equation*}
\omega_k(x_1t_2)= \omega_1(x_k)\cdots \omega_d(x_k),~\forall k\in\{1,\cdots,j\}.
\end{equation*}
Next we need to show there exists a unique $t_{j+1}\in(t_j,1)$ such that
\begin{equation}\label{consistency eq4}
\omega_{j+1}(x_1t_2)= \omega_1(x_{j+1})\cdots \omega_d(x_{j+1}).
\end{equation}
It can be guaranteed by
\begin{equation*}
\omega_1\left(t_j\right)\cdots \omega_d\left(t_j\right)
=t_j\omega\left(t_j\right)\leq t_j\leq (x_1t_2)^{\tau_{j+1}}\leq\omega_{j+1}(x_1t_2).
\end{equation*}
So it suffices to show that
\begin{equation}\label{consistency eq5}
t_j\leq (x_1t_2)^{\tau_{j+1}}.
\end{equation}
Since $x_j=\frac{t_{j-1}}{t_j}$, (\ref{consistency eq5}) is equivalent to show that $x_j\geq \frac{t_{j-1}}{(x_1t_2)^{\tau_{j+1}}}$ which can be implied by
\begin{eqnarray*}
\omega_1\left(\frac{t_{j-1}}{(x_1t_2)^{\tau_{j+1}}}\right)\cdots \omega_d\left(\frac{t_{j-1}}{(x_1t_2)^{\tau_{j+1}}}\right)
&=&\frac{t_{j-1}}{(x_1t_2)^{\tau_{j+1}}}\omega\left(\frac{t_{j-1}}{(x_1t_2)^{\tau_{j+1}}}\right)  \\
&\leq&\frac{t_{j-1}}{(x_1t_2)^{\tau_{j+1}}}\leq (x_1t_2)^{\tau_j}\\
&\leq&\omega_{j}(x_1t_2).
\end{eqnarray*}
So it suffices to show that
\begin{equation*}
t_{j-1}\leq (x_1t_2)^{\tau_j+\tau_{j+1}}.
\end{equation*}
Therefore, the existence of $x_{j+1}$ and $t_{j+1}$ can be guaranteed by
\begin{equation*}
t_2\leq (x_1t_2)^{\tau_2+\cdots\tau_{j+1}}.
\end{equation*}
That is
\begin{equation}\label{consistency eq6}
t_2\leq x_1^{\frac{\tau_2+\cdots\tau_{j+1}}{1-(\tau_2+\cdots+\tau_{j+1}) }}.
\end{equation}
Note that
\begin{equation*}
(x_1t_2)^{\tau_1}\leq\omega(x_1\tau_2)=\omega_1\left(x_1\right)\cdots \omega_d\left(x_1\right)
=x_1\omega\left(x_1\right)\leq x_1.
\end{equation*}
Thus
\begin{equation}
t_2\leq x_1^{\frac{1-\tau_1}{\tau_1}}.
\end{equation}
Since $\tau_2+\cdots\tau_{j+1}\leq\tau_2+\cdots+\tau_{d-1}\leq 1-\tau_1-\tau_d\leq 1-2\tau_1$, we have
\begin{equation}\label{consistency eq7}
\frac{1-\tau_1}{\tau_1}> \frac{\frac{1}{2}-\tau_1}{\tau_1}\geq\frac{\tau_2+\cdots\tau_{j+1}}{1-(\tau_2+\cdots+\tau_{j+1}) }.
\end{equation}
Then combining (\ref{consistency eq6}) and (\ref{consistency eq7}), we obtain (\ref{consistency eq5}). Thus we have showed (\ref{eq3}). Since
\begin{equation*}
\omega(x_1\cdots x_d)\leq C\omega(x_1)\omega(x_2\cdots x_d)\leq\cdots\leq C^{d-1}\omega(x_1)\cdots\omega(x_d),
\end{equation*}
we have
\begin{eqnarray*}
\omega_d(x_1\cdots x_d)&=&\frac{x_1\cdots x_d\omega(x_1\cdots x_d)}{\omega_1(x_1\cdots x_d)\cdots \omega_{d-1}(x_1\cdots x_d)}\\
&=&\frac{x_1\cdots x_d\omega(x_1\cdots x_d)}{\prod_{j=1}^{d-1}\prod_{i=1}^{d}\omega_{i}(x_j)}\\
&=&\frac{x_1\cdots x_d\omega(x_1\cdots x_d)\prod_{i=1}^{d}\omega_i(x_d)}{\prod_{j=1}^{d}\prod_{i=1}^{d}\omega_{i}(x_j)}\\
&\leq&C^{d-1}\frac{x_1\cdots x_d\omega(x_1)\cdots \omega(x_d)\prod_{i=1}^{d}\omega_i(x_d)}{\prod_{j=1}^{d}\prod_{i=1}^{d}\omega_{i}(x_j)}\\
&=&C^{d-1}\omega_1(x_d)\cdots\omega_d(x_d).
\end{eqnarray*}
Note that for $j\in\{1,\cdots,d-1\}$,
\begin{equation*}
\lim_{x_1\rightarrow 0^+}\omega_j(x_1\cdots x_d)=\omega_1(x_j)\cdots\omega_d(x_j)=0.
\end{equation*}
We have $\lim_{x_1\rightarrow 0^+} x_j=0$. By (\ref{comparability}) and $\omega_1(x)\cdots\omega_d(x)=x\omega(x)$, we have $\omega_d(x)\leq x^{\frac{1}{d}}$ for $x$ small enough. Thus if
\begin{eqnarray}
 &&(x_1\cdots x_d)^{\frac{1}{d}}\leq C^{d-1}\omega_1(x_d)\cdots\omega_d(x_d)=C^{d-1} x_d\omega(x_d)\nonumber\\
 &\Leftrightarrow& (x_1\cdots x_{d-1})^{\frac{1}{d}}\leq C^{d-1} x_d^{1-\frac{1}{d}}\omega(x_d)\label{eq4},
\end{eqnarray}
then we still have
\begin{equation*}
\omega_d(x_1\cdots x_d)\leq (x_1\cdots x_d)^{\frac{1}{d}}\leq C^{d-1}\omega_1(x_d)\cdots\omega_d(x_d)=C^{d-1} x_d\omega(x_d).
\end{equation*}
Since $\lim_{x_1\rightarrow 0^+}(x_1\cdots x_{d-1})^{\frac{1}{d}}=0$, we can choose $x_d$ with $\lim_{x_1\rightarrow 0^+} x_d=0$ for which we have (\ref{eq4}) hold and
\begin{equation*}
\omega_j(x_1\cdots x_d)\leq \omega_1(x_j)\cdots\omega_d(x_j),~\forall j\in\{1,\cdots,d-1\}.
\end{equation*}
Hence $\omega_1,\cdots,\omega_d$ satisfy the consistency condition. Therefore, we complete the proof.

\end{proof}

\subsection*{Acknowledgements}
The work is supported by NSFC (No. 11771318, No. 11790274).

\end{document}